\documentclass[11pt]{amsart}
\usepackage{amsmath,amssymb,amsfonts}%
\usepackage{amsthm}%
\usepackage{mathrsfs}%
\usepackage[title]{appendix}%
\usepackage{xcolor}%
\usepackage{manyfoot}
\usepackage[mathlines]{lineno}

\newtheorem{theorem}{Theorem}[section]
\newtheorem{lemma}{Lemma}[section]
\newtheorem{corollary}{Corollary}[section]

\usepackage{xcolor}

\newtheorem{proposition}{Proposition}[section]

\theoremstyle{definition}

\newtheorem{definition}[theorem]{Definition}

\theoremstyle{remark}
\theoremstyle{remark}
\newtheorem{remark}{Remark}[section]

\newtheorem{rule-def}[theorem]{Rule}

\usepackage{microtype}
\usepackage{enumerate}
\usepackage{mismath}

\usepackage{stix}
\usepackage{hyperref}
\usepackage{todonotes}

\allowdisplaybreaks

\author[S. Ashraf]{Salman Ashraf}
\address{Department of Mathematics, Applied Science Cluster, University of Petroleum and Energy Studies (UPES), Dehradun, Uttarakhand, 248007, India}
\email{ashrafsalman869@gmail.com}

\author[H. Rafeiro]{Humberto Rafeiro}
\address{United Arab Emirates University, College of Sciences, Department of Mathematical Sciences, P.O. Box 15551, Al Ain, Abu Dhabi, United Arab Emirates}
\email{rafeiro@uaeu.ac.ae}

\begin{document}

\title[Integral operators on generalized weighted central Morrey  \(\ldots\)]{Integral Operators on Generalized Weighted Central Morrey Spaces over Local Fields}

\begin{abstract}
We introduce generalised weighted central Morrey spaces over local fields and obtain a quantitative estimate for the boundedness of the Hardy--Hilbert-type integral operator on these newly introduced spaces, albeit specifically in the context of power-weighted spaces. A similar estimate is also obtained for the Hardy--Littlewood--P\'olya operator. 
\end{abstract}

\maketitle

\section{Introduction}



The classical Morrey spaces, denoted $M_{p, \lambda}(\mathbb{R}^n)$, were introduced by Morrey \cite{Morrey} and have proven valuable, among other things, in the analysis of the local behaviour of solutions to second-order elliptic partial differential equations \cite{AA1,Morrey,sawano2020morrey}.  Recall that for $0 \leq \lambda \leq n$ and $1 \leq p < \infty$, we say $f \in M_{p, \lambda}$ if 
\begin{align} \label{Introuction1}
\|f\|_{M_{p,\lambda}(\mathbb{R}^n)} := \sup_{x \in \mathbb{R}^n, \, r > 0}  r^{-\lambda/p}\|f\|_{L_p(B(x,r))} < \infty. 
\end{align}
Generalised Morrey spaces $M_{p, \phi}(\mathbb{R}^n)$ are defined by replacing the power term $r^\lambda$ in \eqref{Introuction1} with an appropriate function $\phi: (0, \infty) \rightarrow \mathbb{R}_{>0}$ (see \cite{Morrey3,sawano2020morrey, Morrey1}). Over the past few decades, various fundamental operators have been extensively studied in both classical and generalised Morrey spaces, as well as in generalised local Morrey spaces. For details on the boundedness of some classical operators in harmonic analysis such as maximal operators, fractional integral operators, and singular integral operators on these spaces, we refer the reader to \cite{1,9,10,Raf,11}.

In harmonic analysis, it is important to study weighted estimates for the aforementioned classical operators, among other things, because of their application to PDEs. The boundedness of such operators in weighted Lebesgue spaces was initially established in earlier work (see \cite{W4,W1,W2}). Komori and Shirai \cite{W3} introduced weighted Morrey spaces and examined the boundedness of classical operators in these spaces. In \cite{Morrey3}, Nakamura extended this study to generalised weighted Morrey spaces.



Recall that a local field $K$ is defined as a totally disconnected, locally compact, non-discrete, and complete field. Examples of such fields include $p$-adic fields and fields of formal Laurent series over finite fields, often referred to as $p$-series fields. The classification of local fields is briefly reviewed in Section \ref{sec:2}. Local fields differ from Euclidean spaces as they are totally disconnected and have a non-Archimedean metric, unlike the structure of Euclidean spaces. Since classical differentiation is not applicable in local fields, this limitation makes the study of function space theory in local fields particularly interesting, as it requires the development of alternative analytical tools, such as pseudo-differential operators, to understand smoothness, approximation, and regularity properties in these settings (see \cite{besov4}).

Harmonic analysis on local fields has been developed in numerous studies by Taibleson \cite{T1,T2,T3,MHT}. Furthermore, pseudo-differential equations and stochastic processes over local fields were introduced in \cite{Pseudo}. The theory of $p$-adic analysis was explored by Vladimirov, Volovich, and Zelonov \cite{VV}. Function spaces over local fields help to extend harmonic analysis on local fields, particularly in studying the boundedness of classical operators on them. We refer to the reader \cite{QJ1,QJ2,Ashraf,bib1,BJ3,adic4,adic10,Ho,Nurul,K2,adic1} for the boundedness of classical operators on function spaces such as Lebesgue, Besov, Triebel--Lizorkin and Morrey spaces over local fields and its special case, such as the $p$-adic fields.

Behera \cite{Behera1} recently introduced the Hardy and Hardy--Littlewood--P\'olya operators over a local field $K$ as follows:
\begin{align} \label{Behera1}
\mathcal{H}f(s) = \dfrac{1}{|s|}~\int_{B(0,|s|)} f(t)\di t, \qquad s \in K \setminus \{0\}
\end{align} 
and
\begin{align} \label{Behera2}
\mathcal{T}f(s) = \int_{K} \dfrac{f(t)}{\text{max}\{|s|,|t|\}}\di t, \qquad  s \in K, 
\end{align} 
respectively.

The author obtained boundedness results of \eqref{Behera1} and \eqref{Behera2}  
on weighted Lebesgue spaces $L^p(K, |x|^\alpha \di x).$

In this paper, we have two primary aims. First, motivated by the work of Nakamura \cite{Morrey3}, we introduce  generalised power-weighted central Morrey spaces over local fields, which will be  denoted by $M_{r,\phi}(\omega_\alpha)$ and find conditions on function $\phi: \mathbb{Z} \rightarrow (0, \infty)$ such that the characteristic function of an arbitrary ball belongs to $M_{r,\phi}(\omega_\alpha).$ Second, we define integral operators induced by homogeneous kernel, which serves as a generalization of \eqref{Behera1} and \eqref{Behera2}, see \eqref{H1} and investigate the boundedness of introduced operators on generalized power weighted central Morrey spaces over local fields via mapping properties of dilation operators on these spaces.


This paper is organised as follows. In Section \ref{sec:2}, we give the definition, classification and basic preliminaries of local fields. In Section \ref{sec:3}, we define the generalised power-weighted central Morrey space over local fields $M_{r,\phi}(\omega_\alpha)$. We also give conditions on the function $\phi: \mathbb{Z} \rightarrow (0, \infty)$ such that the characteristic function of $B^k, k \in \mathbb{Z}$, belongs to $M_{r,\phi}(\omega_\alpha).$ We also establish the boundedness of the dilation operator on generalised power weighted central Morrey space over local fields. In Section \ref{sec:4}, we introduce the Hardy--Hilbert-type integral operator $\mathscr{T}$ on a local field. In Section \ref{sec:5}, we compute the operator norm of $\mathscr{T}$ on $M_{r,\phi}(\omega_\alpha).$

\section{Local Fields}\label{sec:2}

Let $K$ be a field equipped with a non-discrete topology such that both its additive group $K^+$ and its multiplicative group $K^*=K\setminus\{0\}$ are locally compact Abelian groups. In this situation, $K$ must either be connected real and complex number fields or totally disconnected.

In this paper, by a local field we mean a field $K$ which is locally compact, non-discrete and totally disconnected. These fields have been thoroughly classified (see \cite{MHT}). The $p$-adic fields are examples of local fields. 

Let $K$ be a local field. Then there is an integer $q=p^c,$ where $p$ is a prime and $c$ is a positive integer, and a natural norm $|\cdot|~ : K \rightarrow \mathbb{R}_{\geq 0}$ such that for every non-zero $x \in K$, $|x| = q^k$ for some integer $k$ and $|0|=0.$ This norm is known as the \textit{absolute value} or \textit{valuation} on $K$ and serves as a non-Archimedean norm on $K,$ that is $|x+y| \leq \text{max} \{|x|, |y|\},$  for all $x,y \in K$ and $|x+y| = \text{max} \{|x|, |y|\}~ \text{if}~ |x| \neq |y|.$ 

For a given  $k \in \mathbb{Z},$ we introduce the following:
\begin{equation*}
B^k := \{y \in K \mid |y| \leq q^{-k}\}, \quad 
 S^k := B^k\setminus B^{k+1}= \{y \in K \mid |y| = q^{-k}\},
\end{equation*} 
where $B^k$ is the ball centred at $0$ with radius $q^{-k}$ and $S^k$ its boundary. 

Let $\di x$ be the Haar measure on the locally compact Abelian group $K^{+}.$  For a measurable subset $E$ of $K$, $|E|$ denote the Haar measure of $E,$ i.e., $|E|=\int_K \chi_E(x)\di x$, where $\chi_E$ is the characteristic function of $E.$ We normalise the Haar measure such that $|B^0|=|\{y \in K \mid |y| \leq 1\}| = 1$. With respect to this normalization, we have
\begin{align*}
|B^k|=q^{-k}, \qquad |S^k|=q^{-k}(1-q^{-1}). 
\end{align*}

For $\alpha \in K\setminus \{0\}$, we have  $\di (\alpha x) =|\alpha| \di x.$  Also note that, if $f \in L^1(K),$ the integral can be expressed as
\begin{align*}
\int_{K} f(x)\di x = \sum\limits_{\gamma=-\infty}^{\infty} \int_{S^\gamma} f(x)\di x.
\end{align*}
Moreover, the change of variable formula in $K$ is given by:
\begin{align} \label{Change1}
\int_{K} f(x)\di x = |y| \int_{K} f(yx)\di x, \quad y \in K \setminus \{0\}.
\end{align}
For a more detailed discussion and proofs of the results covered in this section, we refer to the book by Taibleson \cite{MHT} and the work \cite{BJ3}.

\section{Generalized weighted central Morrey spaces on local fields} \label{sec:3}
A \textit{weight function} (or a \textit{weight}) $\omega$ on $K$ is a non-negative, locally integrable function on $K.$ For a weight $\omega$ and a measurable subset $A \subset K,$ we denote the weight of $A$ by:
\begin{align*}
\omega(A) \coloneqq \int_{A} \omega(x)\di x.
\end{align*} 

Now we introduce generalised weighted central Morrey spaces with power weights on local fields.
\begin{definition} \label{Gen_Morrey}
Let $1\leq r < \infty$, $\phi: \mathbb Z  \to \mathbb R_{>0}$ be a function, and 
$\omega_\alpha(x) = |x|^\alpha$, $\alpha >0$, be the power function. The \emph{generalized weighted central Morrey space} with weight $\omega_\alpha$ over the local field $K$, denoted by  $M_{r,\phi}(\omega_\alpha)$, is defined as the set of all $f \in L^r_{\mathrm{loc}}(K)$ such that 
\begin{align*}
\|f\|_{M_{r,\phi}(\omega_\alpha)} = \sup_{k \in \mathbb{Z}}  \phi(k)\Bigg( \dfrac{1}{|B^k|} \int_{B^k}|f(y)|^r \omega_\alpha(y) \di y\Bigg)^{\frac{1}{r}} < \infty.
\end{align*}
\end{definition}

We recover the power weighted central Morrey space $M^r_t(\omega_\alpha)$ and the Lebesgue space $L^r(\omega_\alpha)$ under appropriate $\phi(k)$, namely:
\begin{align*}
M_{r,\phi}(\omega_\alpha) = \begin{cases}
M_t^r(\omega_\alpha), \quad &~\text{if}~~~\phi(k)=|B^k|^{\frac{1}{t}},\\
L^r(\omega_\alpha),\quad &~\text{if}~~~ \phi(k)=|B^k|^{\frac{1}{r}}.
\end{cases}
\end{align*}


Inclusion of the characteristic function $\chi_{B}$ of a ball $B$ in a function space is essential to analyse the structure and properties of function spaces.
However, computing the norm $\|\chi_{B^k}\|_{M_{r,\phi}(\omega_\alpha)}$ for the characteristic function $\chi_{B^k}$ of a ball $B^k$  presents significant challenges. To address this difficulty, we impose the following condition on $\phi$ throughout this paper:

\begin{definition}
Let $1 \leq r < \infty$, $\alpha >0$, and $q$ be a prime power. 
We say that a function $\phi: \mathbb{Z} \rightarrow \mathbb{R}_{>0}$ belongs to the class $\Phi_{r,q}$ if there exists $C > 0$ such that: 
\begin{align} \label{Condition1}
\phi(k) \leq C, \quad \text{if } k \in \mathbb N
\end{align}
and
\begin{align} \label{Condition2}
\phi(k) \leq C q^{-k/r}, \quad \text{if } k \in \mathbb{Z}\setminus \mathbb{N}.
\end{align}
\end{definition}

\begin{lemma} \label{char_norm}
Let $K$ be a local field,  $1 \leq r < \infty$, $\alpha >0$, and $\phi \in \Phi_{r,q}$, where $q=|B|^{-1}$. 
Then 
\begin{align*}
 \|\chi_{B^{\eta}}\|_{M_{r,\phi}(\omega_\alpha)}   \leq  C \bigg(\dfrac{\omega_\alpha(B^{\eta})}{|B^{\eta}|}\bigg)^{\frac{1}{r}}\max\{1, q^{-\eta/r}\},
\end{align*}
where $C>0$ and $\chi_{B^{\eta}}$ is the characteristic function of $B^{\eta}$, $\eta \in \mathbb Z$.
\end{lemma}
\begin{proof}
By the Definition \ref{Gen_Morrey}, we have 
\begin{align*}
\|\chi_{B^{\eta}}\|_{M_{r,\phi}(\omega_\alpha)} &= \sup_{k \in \mathbb{Z}}  \phi(k)\Bigg( \dfrac{1}{|B^k|} \int_{B^k}|\chi_{B^{\eta}}(y)|^r \omega_\alpha(y) \di y\Bigg)^{\frac{1}{r}} \\
& = \sup_{k \in \mathbb{Z}}  \phi(k)\Bigg( \dfrac{ \omega_\alpha(B^{\eta} \cap B^k)}{|B^k|}\Bigg)^{\frac{1}{r}}.
\end{align*}

There are four cases to consider:

\medskip 

Case 1 ($k \in \mathbb{N}$ and $k > \eta$): Since 
 $B^k \subset B^{\eta}$, we have 
\begin{align} 
\|\chi_{B^{\eta}}\|_{M_{r,\phi}(\omega_\alpha)} = \sup_{k \in \mathbb{Z}}  \phi(k)\Bigg( \dfrac{ \omega_\alpha(B^k)}{|B^k|}\Bigg)^{\frac{1}{r}}.
\end{align}

Also
\begin{align} \nonumber
\omega_\alpha(B^k) &= \int_{B^k} |y|^\alpha \di y \\ \nonumber
&= \sum_{l=-\infty}^{-k} \int_{\{y : |y| = q^l\}} |y|^\alpha \di y\\ \nonumber
&= \sum_{l=-\infty}^{-k} q^{\alpha l}|\{y : |y| = q^l\}| \\ \nonumber
&= \sum_{l=-\infty}^{-k} q^{\alpha l}(q^l-q^{l-1}) \\  \label{power1}
&= \dfrac{(q-1)q^{\alpha-k(\alpha+1)}}{q^{\alpha+1}-1}. 
\end{align}

Therefore, by \eqref{Condition1} and \eqref{power1}, we have
\begin{align} \label{pf1}
\phi(k)\Bigg( \dfrac{ \omega_\alpha(B^k)}{|B^k|}\Bigg)^{\frac{1}{r}} \leq C \bigg(\dfrac{\omega_\alpha(B^{\eta})}{|B^{\eta}|}\bigg)^{\frac{1}{r}}.
\end{align}

Case 2 ($k \in \mathbb N$ and $\eta >k$): The inclusion   
 $B^{\eta} \subset B^{k}$ and \eqref{Condition1} yield 
\begin{align} \label{pf11}
 \phi(k)\Bigg( \dfrac{ \omega_\alpha(B^{\eta})}{|B^k|}\Bigg)^{\frac{1}{r}} \leq C \bigg(\dfrac{\omega_\alpha(B^{\eta})}{|B^{\eta}|}\bigg)^{\frac{1}{r}}.
\end{align}

 Case 3 ($k \in \mathbb{Z} \setminus \mathbb{N}$ and $\eta >k$): By  \eqref{Condition2} and the fact that $B^{\eta} \cap B^k= B^{\eta}$,   we have
\begin{align} \label{pf2}
\phi(k)\Bigg( \dfrac{ \omega_\alpha(B^{\eta})}{|B^k|}\Bigg)^{\frac{1}{r}} \leq C \omega_\alpha(B^{\eta})^{\frac{1}{r}}.
\end{align}

Case 4  ($k \in \mathbb{Z} \setminus \mathbb{N}$ and $k > \eta$): By \eqref{Condition2} we obtain
\begin{align} \label{pf22}
\phi(k)\Bigg( \dfrac{ \omega_\alpha(B^{k})}{|B^k|}\Bigg)^{\frac{1}{r}} \leq C \omega_\alpha(B^{\eta})^{\frac{1}{r}}.
\end{align}

Therefore, by \eqref{pf1}-\eqref{pf22}, we get the estimate
\begin{align*}
\|\chi_{B^{\eta}}\|_{M_{r,\phi}(\omega_\alpha)} &\leq  C \max\Bigg\{\bigg(\dfrac{\omega_\alpha(B^{\eta})}{|B^{\eta}|}\bigg)^{\frac{1}{r}},~\omega_\alpha(B^{\eta})^{\frac{1}{r}}\Bigg\} \\
& \leq  C \bigg(\dfrac{\omega_\alpha(B^{\eta})}{|B^{\eta}|}\bigg)^{\frac{1}{r}}\max\{1, q^{-\eta/r}\},
\end{align*}
which completes the proof.
\end{proof}

\begin{remark}
    The norm estimate for characteristic functions of ball $\chi_{B^{\eta}}$ in $M_{r,\phi}(\omega_\alpha)$ implies that all locally compactly supported functions $C_c(K)$ belong to $M_{r,\phi}(\omega_\alpha).$ This follows because any $f \in C_c(K)$ can be approximated by finite linear combinations of $\chi_{B^{\eta}}$ and the assumption 
    $$\bigg(\dfrac{\omega_\alpha(B^{\eta})}{|B^{\eta}|}\bigg)^{\frac{1}{r}}\max\{1, q^{-\eta/r}\} < \infty$$
    in Lemma \ref{char_norm}, allows us to view the space of locally compactly supported functions $C_c(K)$ as a natural subspace of $M_{r,\phi}(\omega_\alpha).$
\end{remark}

As a consequence of Lemma \ref{char_norm}, we will show the function $\langle y \rangle^{-N}$ where $\langle y \rangle = \max\{1, |y|\}$ belongs to $M_{r,\phi}(\omega_\alpha).$

\begin{corollary}
Let $K$ be a local field,  $1 \leq r < \infty$, $\alpha >0$, and $\phi \in \Phi_{r,q}$, where $q=|B|^{-1}$. If $N > \dfrac{\alpha+1}{r},$ then  $\langle y \rangle^{-N} \in M_{r,\phi}(\omega_\alpha).$
\end{corollary}

\begin{proof}
The function $\langle y \rangle^{-N}$ can be written as:
\begin{align*}\langle y \rangle^{-N}=\begin{cases}
1, & \text{if } |y| \leq 1, \\
|y|^{-N}, & \text{if } |y| > 1.
\end{cases}
\end{align*}

So we can write:
\begin{align}
    \|\langle y \rangle^{-N}\|_{M_{r,\phi}(\omega_\alpha)} \leq  \|\chi_{B^{0}}\|_{M_{r,\phi}(\omega_\alpha)}  ~+ ~\||y|^{-N}\chi_{K \setminus B^{0}}\|_{M_{r,\phi}(\omega_\alpha)}.
\end{align}

By Lemma \ref{char_norm}, we have 
\begin{align} \label{max_func_1}
    \|\chi_{B^{0}}\|_{M_{r,\phi}(\omega_\alpha)} \leq  C \bigg(\dfrac{\omega_\alpha(B^{0})}{|B^{0}|}\bigg)^{\frac{1}{r}}\max\{1, q^{-0/r}\} = C \omega_\alpha(B^{0})^{\frac{1}{r}}.
\end{align}

For $K \setminus B^0,$ we use the decomposition:
\begin{align*}
    K \setminus B^0 = \bigcup_{\eta=1}^{\infty} S^{-\eta}.
\end{align*}
So,
\begin{align*}
    |y|^{-N}\chi_{K \setminus B^{0}} = \sum_{\eta=1}^{\infty} q^{-\eta N} \chi_{S^{-\eta}}.
\end{align*}
By Lemma \ref{char_norm}, for each $\eta \geq 1$, we have
\begin{align}
    \|\chi_{B^{-\eta}}\|_{M_{r,\phi}(\omega_\alpha)} \leq  C \bigg(\dfrac{\omega_\alpha(B^{-\eta})}{|B^{-\eta}|}\bigg)^{\frac{1}{r}}\max\{1, q^{\eta/r}\} = C\omega_\alpha(B^{-\eta})^{\frac{1}{r}}.
\end{align}
Since $S^{-\eta} \subset B^{-\eta},$ we have $\chi_{S^{-\eta}}(x) \leq \chi_{B^{-\eta}}(x).$ Hence,
\begin{align} \nonumber
    \||y|^{-N}\chi_{K \setminus B^{0}}\|_{M_{r,\phi}(\omega_\alpha)} & \leq \sum_{\eta=1}^{\infty} q^{-\eta N} \|\chi_{B^{-\eta}}\|_{M_{r,\phi}(\omega_\alpha)} \\  \label{max_func_2}
    & \leq C \sum_{\eta=1}^{\infty} q^{-\eta N} \omega_\alpha(B^{-\eta})^{\frac{1}{r}}.
\end{align}

Now, by combining the estimates, \eqref{max_func_1} and \eqref{max_func_2}, we have 
\begin{align*} 
    \|\langle y \rangle^{-N}\|_{M_{r,\phi}(\omega_\alpha)} & \leq  C \bigg(\omega_\alpha(B^{0})^{\frac{1}{r}}  ~+ ~\sum_{\eta=1}^{\infty} q^{-\eta N} \omega_\alpha(B^{-\eta})^{\frac{1}{r}}\bigg) \\
    & \leq C \bigg[\bigg(\dfrac{(q-1)q^{\alpha}}{q^{\alpha+1}-1}\bigg)^{\frac{1}{r}}  ~+ ~\sum_{\eta=1}^{\infty} q^{-\eta N} \bigg(\dfrac{(q-1)q^{\alpha+\eta(\alpha+1)}}{q^{\alpha+1}-1}\bigg)^{\frac{1}{r}}\bigg] \\ 
    & \leq C\bigg(\dfrac{(q-1)q^{\alpha}}{q^{\alpha+1}-1}\bigg)^{\frac{1}{r}} \bigg[ 1 ~+ ~\sum_{\eta=1}^{\infty} q^{-\eta(N-\frac{\alpha+1}{r})}\bigg] < \infty.
\qedhere
\end{align*}
\end{proof}

\subsection{Dilation operator in generalised weighted central Morrey spaces}

Let $\tau \in K^* $ and, for any function $f$ on $K,$ consider the dilation operator of the form 
\begin{equation*}
(\mathcal{D}_\tau f)(x) = f(\tau x),\qquad x \in K.
\end{equation*}
Recall that a function $\phi$ is said to be \emph{submultiplicative} if
\begin{equation}\label{submultiplicative}
\phi(s+t) \leq C \phi(s)\phi(t).
\end{equation}
The class of functions $\phi \in \Phi_{r,q}$ satisfying additionally the condition \eqref{submultiplicative} is denoted by $ \Phi_{r,q}^{\textrm{sm}}$.

It should be noted that submultiplicativity is often defined differently by various authors; we follow, e.g., \cite{Kyprianou2006-qe}. For an alternative definition, see \cite{sawano2020morrey}.

\begin{lemma} \label{Dilation}
Let $K$ be a local field, $1 < r < \infty$, $\alpha >0$,  and $ \phi \in \Phi_{r,q}^{\emph{sm}}$, where $q=|B|^{-1}$. 
Then there exists $C>0$ such that
\begin{align}\label{111}
\|\mathcal{D}_\tau f\|_{M_{r,\phi}(\omega_\alpha)} & \leq C |\tau|^{-(1+\alpha)/r} \|f\|_{M_{r,\phi}(\omega_\alpha)}, \quad \text{for all} ~ 0< |\tau| < 1,
\end{align}
and
\begin{align}\label{222}
\|\mathcal{D}_\tau f\|_{M_{r,\phi}(\omega_\alpha)} & \leq  C |\tau|^{-\alpha/r} \|f\|_{M_{r,\phi}(\omega_\alpha)}, \quad \text{for all} ~ 1 \leq |\tau| < \infty.
\end{align}
\end{lemma}

\begin{proof}
We know that $K^*$ can be written as a disjoint union $\bigcup_{l=-\infty}^{\infty}\{y \mid |y| = q^{-l}\}$ and, since $\tau \in K^*$,  $|\tau| = q^{-l}$ for some $l \in \mathbb{Z}.$  According to Definition \ref{Gen_Morrey}, we obtain
\begin{align*}
\|\mathcal{D}_\tau f\|_{M_{r,\phi}(\omega_\alpha)} &= \sup_{k \in \mathbb{Z}}  \phi(k)\Bigg( \dfrac{1}{|B^k|} \int_{B^k}|f(\tau y)|^r |y|^\alpha \di y\Bigg)^{\frac{1}{r}} \\
& = \sup_{ k \in \mathbb{Z}}  \phi(k)\Bigg( \dfrac{|\tau|^{-1} |\tau|^{-\alpha}}{|B^k|} \int_{B^{k+l}}|f(y)|^r |y|^\alpha \di y\Bigg)^{\frac{1}{r}} \\
& = \sup_{k \in \mathbb{Z}}  \phi(k+l-l)\Bigg( \dfrac{|\tau|^{-\alpha}}{|B^{k+l}|} \int_{B^{k+l}}|f(y)|^r |y|^\alpha \di y\Bigg)^{\frac{1}{r}} \\
 & \leq \sup_{k \in \mathbb{Z}}  \phi(k+l)\phi(-l)\Bigg( \dfrac{|\tau|^{-\alpha}}{|B^{k+l}|} \int_{B^{k+l}}|f(y)|^r |y|^\alpha \di y\Bigg)^{\frac{1}{r}} \\
& = \phi(-l) |\tau|^{-\alpha/r} \|f\|_{M_{r,\phi}(\omega_\alpha)} .
\end{align*}
Using the previous estimate along with equations \eqref{Condition1} and \eqref{Condition2}, we obtain \eqref{111} and \eqref{222}.
\end{proof}

\section{Hardy--Hilbert-type integral operators on local fields} \label{sec:4}

The Hardy--Hilbert-type integral operator $ \mathscr{T} $, on a local field $K$, is defined as:
\begin{align} \label{H1}
\mathscr{T}f(s) = \int_{K^*} \mathscr{K}(|s|,|t|)f(t)\di t,~~~~  s \in K^*
\end{align} 
where 
$\mathscr{K}(s,t)$ is a non-negative measurable function on $(0, \infty) \times (0, \infty)$ that satisfies the condition
\begin{align} \label{ch4_homo}
\mathscr{K}(\xi s,\xi t) = \xi^{-1} \mathscr{K}(s,t),   \quad \xi >0,
\end{align}
which means that it is homogeneous of degree -1.

Examples of such operators are:

\begin{description}
\item[Hardy operator] Choosing 
\begin{align*}
\mathscr{K}(|s|, |t|) = |s|^{-1} \chi_{B(0,|s|)}(t),~~~~~~~~~~s \in K^*
\end{align*}
where $\chi_{B(0,|s|)}$ is the characteristic function of the ball $ B(0,|s|) = \{ t \in K \mid |t| \leq |s|\},$  we get the Hardy operator on $K$
\begin{align} \label{ch4_Hardy}
\mathcal{H}f(s) = \dfrac{1}{|s|}~\int_{B(0,|s|)} f(t)\di t,~~~~~~~~~~~  s \in K^*.
\end{align} 
\item[Hilbert operator] Setting 
\begin{align*}
\mathscr{K}(|s|,|t|) = \dfrac{1}{|s| + |t|},
\end{align*}
we get the Hilbert operator on $K$ 
\begin{align} \label{ch4_Hilbert}
\mathscr{H}f(s) = \int_{K^*} \dfrac{f(t)}{|s| + |t|}\di t,~~~~~~~~  s \in K^*.
\end{align}
\item[Hardy--Littlewood--P\'{o}lya operator] Putting 
\begin{align*}
\mathscr{K}(|s|,|t|) = \dfrac{1}{\text{max}\{|s|,|t|\}},~~~~s \in K^*,
\end{align*}
we obtain the Hardy--Littlewood--P\'{o}lya operator on $K$, given by
\begin{align} \label{Hardy_Littlewood}
\mathscr{P}f(s) = \int_{K^*} \dfrac{f(t)}{\text{max}\{|s|,|t|\}}\di t,~~~~~~~~  s \in K^*. 
\end{align} 
\end{description}

A function $f$ in the local field $K$ is said to be \textit{radial} if $f(x)=f(y)$ whenever $|x|=|y|, ~ x, y \in K.$ In the following proposition, we establish that the norm of the operator $\mathscr{T}$ on the space $M_{r,\phi}(\omega_\alpha),$ is equivalent to the norm of $\mathscr{T}$ when restricted to the subspace of radial functions.

\begin{proposition}
The operator norm of $\mathscr{T} $ on the space $M_{r,\phi}(\omega_\alpha)$ is equivalent to the norm of its restriction to radial functions.
\end{proposition}
\begin{proof}
Let $f$ be a function on $K^*.$ Define $\bar{f}(x)$ as the average value of $f$ over the sphere with radius $|x|$ centred at $0.$ Formally, this can be written as:
\begin{align*}
\bar{f}(x)=\dfrac{1}{|x|(1-q^{-1})}\int_{S(0,|x|)} f(y) \di y, \qquad x \in K^*.
\end{align*}
By definition $\bar{f}$ is a radial function. Now
\begin{align*}
\mathscr{T}\bar{f}(x) &= \int_{K^*} \mathscr{K}(|x|,|y|)\bar{f}(y)\di y \\
&=\int_{K^*} \mathscr{K}(|x|,|y|)\Bigg(\dfrac{1}{|y|(1-q^{-1})}\int_{|t|=|y|} f(t) \di t \Bigg)\di y \\
&=\int_{K^*} \mathscr{K}(|x|,|y|)\Bigg(\dfrac{1}{(1-q^{-1})}\int_{|t|=1} f(ty) \di t \Bigg)\di y \\
&=\dfrac{1}{(1-q^{-1})}\int_{|t|=1} \Bigg(\int_{K^*} \mathscr{K}(|x|,|y|)f(ty) \di y\Bigg) \di t \\
&=\dfrac{1}{(1-q^{-1})}\int_{|t|=1} \Bigg(|t|^{-1}\int_{K^*} \mathscr{K}(|x|,|t^{-1}y|)f(y) \di y\Bigg) \di t \\
&=\dfrac{1}{(1-q^{-1})}\int_{|t|=1} \Bigg(\int_{K^*} \mathscr{K}(|tx|,|y|)f(y) \di y\Bigg) \di t \\
&=\dfrac{1}{(1-q^{-1})}\int_{|t|=1} \mathscr Tf(tx) \di t.
\end{align*}
Since $\mathscr{T}f(tx)=\mathscr{T}f(x)$ for all $t \in D^*,$ where $D^*= \{x \in K : |x|=1\}.$  Hence
\begin{align}\label{samegandf}
\mathscr{T}\bar{f}(x)=\dfrac{1}{(1-q^{-1})}\int_{|t|=1} \mathscr Tf(x) \di t = \mathscr{T}f(x).
\end{align}
By Holder's inequality and \eqref{Change1}, we have
\begin{align*}
\|\bar{f}\|_{M_{r,\phi}(\omega_\alpha)}& =\sup_{k \in \mathbb{Z}}  \phi(k)\Bigg( \dfrac{1}{|B^k|} \int_{B^k}\Bigg|\dfrac{1}{|y|(1-q^{-1})}\int_{S(0,|y|)} f(t) \di t\Bigg|^r |y|^\alpha\di y\Bigg)^{\frac{1}{r}}\\
& \leq \sup_{k \in \mathbb{Z}}\phi(k)\Bigg( \dfrac{1}{|B^k|} \int_{B^k}\Bigg(\int_{S(0,|y|)} |f(t)|^r \di t\Bigg)|y|^{\alpha-1} (1-q^{-1})^{-1}\di y\Bigg)^{\frac{1}{r}}\\
& = \sup_{k \in \mathbb{Z}}\dfrac{\phi(k)}{(1-q^{-1})^{\frac{1}{r}}}\Bigg( \dfrac{1}{|B^k|} \int_{B^k}\Bigg(\int_{|t|=|y|} |f(t)|^r |y|^{-1} \di t\Bigg)|y|^{\alpha} \di y\Bigg)^{\frac{1}{r}}\\
& = \sup_{ k \in \mathbb{Z}}\dfrac{\phi(k)}{(1-q^{-1})^{\frac{1}{r}}}\Bigg( \dfrac{1}{|B^k|} \int_{B^k}\Bigg(\int_{|t|=1} |f(ty)|^r\di t\Bigg)|y|^{\alpha} \di y\Bigg)^{\frac{1}{r}}\\
& = \sup_{ k \in \mathbb{Z}}\dfrac{\phi(k)}{(1-q^{-1})^{\frac{1}{r}}}\Bigg( \dfrac{1}{|B^k|} \int_{|t|=1}\Bigg(\int_{B^k} |f(ty)|^r |y|^{\alpha} \di y\Bigg)\di t\Bigg)^{\frac{1}{r}}\\
& = \sup_{ k \in \mathbb{Z}}\dfrac{\phi(k)}{(1-q^{-1})^{\frac{1}{r}}}\Bigg[\int_{|t|=1}\di t\Bigg]^{\frac{1}{r}} \Bigg( \dfrac{1}{|B^k|}\int_{B^k} |f(y)|^r |y|^{\alpha} \di y\Bigg)^{\frac{1}{r}}\\
& = \sup_{ k \in \mathbb{Z}}\phi(k)\Bigg( \dfrac{1}{|B^k|}\int_{B^k} |f(y)|^r |y|^{\alpha} \di y\Bigg)^{\frac{1}{r}}\\
&=\|f\|_{M_{r,\phi}(\omega_\alpha)}
\end{align*}
which, combined with \eqref{samegandf}, yields
\begin{align*}
\dfrac{\|\mathscr{T}f\|_{M_{r,\phi}(\omega_\alpha)}}{\|f\|_{M_{r,\phi}(\omega_\alpha)}} \leq \dfrac{\|\mathscr{T}\bar{f}\|_{M_{r,\phi}(\omega_\alpha)}}{\|\bar{f}\|_{M_{r,\phi}(\omega_\alpha)}}
\end{align*}
whereby the result follows.
\end{proof}

\section{Quantitative boundedness result}\label{sec:5}

\begin{theorem} \label{Main_Thm}
Let $K$ be a local field, $1 < r < \infty$, $\alpha >0$,  and $ \phi \in \Phi_{r,q}^{\emph{sm}}$, where $q=|B|^{-1}$.
If the kernel $\mathscr{K}$ satisfy
\begin{multline} \label{Main_Thm_1}
C_{r,q} \\ := (1-q^{-1})\Bigg[\mathscr{K}(1,1)~+~\sum_{l=1}^{\infty}\Bigg(\mathscr{K}(1,q^{l})~q^{l(1-\alpha/r)}~+~\mathscr{K}(1,q^{-l})~q^{l(\alpha/r + 1/r -1)}\Bigg)  \Bigg] < \infty,
\end{multline}
then
\begin{align*}
\|\mathscr{T}f\|_{M_{r,\phi}(\omega_\alpha)} \leq C_{r,q} \|f\|_{M_{r,\phi}(\omega_\alpha)},
\end{align*}
where $\mathscr{T}$ is the operator in \eqref{H1}.
\end{theorem}

\begin{proof}
After the change of variable   $y=x\tau$  and  using the relation $\di y=|x| \di \tau$, we obtain
\begin{align*}
\mathscr{T}f(x) & = \int_{K^*} \mathscr{K}(|x|, |x\tau|)f(\tau x)|x|\di \tau \\ 
& =  \int_{K^*} |x|^{-1}\mathscr{K}(1,|\tau| )f(\tau x)|x|\di \tau \\
& =  \int_{K^*}\mathscr{K}(1,|\tau| )f(\tau x)\di \tau \\ 
& = \sum_{l=-\infty}^{\infty}  \int_{S^l} \mathscr{K}(1,|\tau| )\mathcal{D}_\tau f(x)\di \tau \\
& = \sum_{l=-\infty}^{\infty}\mathscr{K}(1,q^{-l})~\int_{S^l}\mathcal{D}_\tau f(x)\di \tau. 
\end{align*}

Applying $\|\cdot\|_{M_{r,\phi}(\omega_\alpha)}$ norm and Minkowski's integral inequality, we get  
\begin{align*}
&\|\mathscr{T}f\|_{M_{r,\phi}(\omega_\alpha)}  \leq \sum_{l=-\infty}^{\infty}\mathscr{K}(1,q^{-l}) \sup_{k \in \mathbb{Z}}  \phi(k)\Bigg( \dfrac{1}{|B^k|} \int_{B^k}\Bigg|\int_{S^l}\mathcal{D}_\tau f(x)\di \tau\Bigg|^r |x|^\alpha\di x\Bigg)^{\frac{1}{r}}\\
& \leq \sum_{l=-\infty}^{\infty}\mathscr{K}(1,q^{-l}) \int_{S^l} \Bigg(\sup_{k \in \mathbb{Z}}  \phi(k)\Bigg( \dfrac{1}{|B^k|} \int_{B^k}|\mathcal{D}_\tau f(x)|^r |x|^\alpha\di x \Bigg)^{\frac{1}{r}} \Bigg) \di \tau\\
& = \sum_{l=-\infty}^{\infty}\mathscr{K}(1,q^{-l})~\int_{S^l}\|\mathcal{D}_\tau f(x)\|_{M_{r,\phi}(\omega_\alpha)}\di \tau \\
& = \sum_{l=-\infty}^{0}\mathscr{K}(1,q^{-l})~\int_{S^l}\|\mathcal{D}_\tau f(x)\|_{M_{r,\phi}(\omega_\alpha)}\di \tau +\sum_{l=1}^{\infty}\mathscr{K}(1,q^{-l})~\int_{S^l}\|\mathcal{D}_\tau f(x)\|_{M_{r,\phi}(\omega_\alpha)}\di \tau \\
& \leq \sum_{l=-\infty}^{0}\mathscr{K}(1,q^{-l})~\int_{S^l}C|\tau|^{-\alpha/r}\|f\|_{M_{r,\phi}(\omega_\alpha)}\di \tau +\sum_{l=1}^{\infty}\mathscr{K}(1,q^{-l})~\int_{S^l} C|\tau|^{-(1+\alpha)/r} \|f\|_{M_{r,\phi}(\omega_\alpha)}\di \tau \\
& = \sum_{l=0}^{\infty}\mathscr{K}(1,q^{l})~\int_{S^{-l}}C |\tau|^{-\alpha/r} \|f\|_{M_{r,\phi}(\omega_\alpha)}\di \tau +\sum_{l=1}^{\infty}\mathscr{K}(1,q^{-l})~\int_{S^l} C|\tau|^{-(1+\alpha)/r} \|f\|_{M_{r,\phi}(\omega_\alpha)}\di \tau \\
& = C\Bigg[\sum_{l=0}^{\infty}\mathscr{K}(1,q^{l})~|S^{-l}| q^{-l\alpha/r}\|f\|_{M_{r,\phi}(\omega_\alpha)} +\sum_{l=1}^{\infty}\mathscr{K}(1,q^{-l})~|S^l|q^{l(1+\alpha)/r} \|f\|_{M_{r,\phi}(\omega_\alpha)} \Bigg] \\
& = C(1-q^{-1})\Bigg[\mathscr{K}(1,1)~+~\sum_{l=1}^{\infty}\mathscr{K}(1,q^{l})~q^lq^{-l\alpha/r}  ~+~\sum_{l=1}^{\infty}\mathscr{K}(1,q^{-l})~q^{-l}q^{l(1+\alpha)/r}  \Bigg]\|f\|_{M_{r,\phi}(\omega_\alpha)} \\
& = C(1-q^{-1})\Bigg[\mathscr{K}(1,1)~+~\sum_{l=1}^{\infty}\Bigg(\mathscr{K}(1,q^{l})~q^{l(1-\alpha/r)}~+~\mathscr{K}(1,q^{-l})~q^{l(\alpha/r + 1/r -1)}\Bigg)  \Bigg]\|f\|_{M_{r,\phi}(\omega_\alpha)} \\
&= C_{r,q} \|f\|_{M_{r,\phi}(\omega_\alpha)},
\end{align*} 
which ends the proof.
\end{proof}

\section{A Particular Result}
In this section, we will obtain the Hardy--Littlewood--P\'olya inequality on the $M_{r,\phi}(\omega_\alpha),$ by taking the particular kernel $\mathscr{K}$ in the operator $\mathscr{T}$ defined by \eqref{H1}. 
\begin{theorem}
Suppose $1 < r < \infty, ~ \alpha >0$ and $ \phi \in \Phi_{r,q}^{\emph{sm}}.$ If $\alpha + 1 < r,$ then the operator defined $\mathscr{P},$ defined by \eqref{Hardy_Littlewood}, is bounded on $M_{r,\phi}(\omega_\alpha).$
\end{theorem}

\begin{proof}
Let $\mathscr{K}(s,t) = \frac{1}{\max\{|s|,~|t|\}}.$ Then the result follows, according to Theorem \ref{Main_Thm}, if $C_{r,q} < \infty.$ We have 
\begin{align*}
C_{r,q} & = (1-q^{-1})\Bigg[\mathscr{K}(1,1)~+~\sum_{l=1}^{\infty}\Bigg(\mathscr{K}(1,q^{l})~q^{l(1-\alpha/r)}~+~\mathscr{K}(1,q^{-l})~q^{l(\alpha/r + 1/r -1)}\Bigg)  \Bigg] \\
& = (1-q^{-1})\Bigg[1~+~\sum_{l=1}^{\infty}\Bigg(\dfrac{q^{l(1-\alpha/r)}}{q^l}~+~q^{l(\alpha/r + 1/r -1)}\Bigg)  \Bigg] \\
& = (1-q^{-1})\Bigg[1~+~\sum_{l=1}^{\infty}\Bigg(q^{-l\alpha/r}~+~q^{l(\alpha/r + 1/r -1)}\Bigg)  \Bigg].
\end{align*}

Because $\alpha > 0$ and $\alpha + 1 < r,$ it follows that both the geometric series $\sum_{l=1}^{\infty}q^{-l\alpha/r} $ and $\sum_{l=1}^{\infty}q^{l(\alpha/r + 1/r -1)}$ are convergent. Therefore, we get the finiteness of $C_{r,q}.$
\end{proof}

\bibliographystyle{amsplain}

\end{document}